\documentclass[reqno,12pt]{amsart}
\theoremstyle{plain}    
\newtheorem{thm}{Theorem}[section]
\newtheorem{lemma}[thm]{Lemma} 
\newtheorem{prop}[thm]{Proposition}

\theoremstyle{remark}

\newtheorem{remark}[thm]{Remark}

\theoremstyle{definition}

\newtheorem{ques}[thm]{Question}

\usepackage{amscd,amssymb,comment,epic,eepic,euscript,graphicx}
\usepackage{enumerate}

\newcommand\Cpx{{\mathbf C}}
\newcommand\cycsim{\overset{\mathrm{cyc}}{\sim}}
\newcommand\HEu{{\EuScript H}}                   
\newcommand\Hh{{\widehat H}}
\newcommand\Nats{{\mathbf N}}
\newcommand\RealPart{{\mathrm{Re}\;}}
\newcommand\Reals{{\mathbf R}}
\newcommand\Tr{{\mathrm{Tr}}}

\begin{document}

\title[BMV conjecture]{Sum--of--squares results for polynomials related to the Bessis--Moussa--Villani conjecture}

\author[Collins]{Beno\^\i{}t Collins$^{\dagger}$}
\address{B.\ Collins,
Department of Mathematics and Statistics, University of Ottawa,
585 King Edward,
Ottawa, ON
K1N 6N5 Canada, and
CNRS, Department of Mathematics, Lyon 1 Claude Bernard University, France} 
\email{bcollins@uottawa.ca}
\thanks{\footnotesize $^{\dagger}$Research supported in part by NSERC
grant RGPIN/341303--2007 and by the
Granma and Galoisint ANR grants}

\author[Dykema]{Kenneth J. Dykema$^{*}$}
\address{K.\ Dykema and F.\ Torres-Ayala, Department of Mathematics, Texas A\&M University,
College Station, TX 77843-3368, USA}
\email{kdykema@math.tamu.edu, francisc@math.tamu.edu}
\thanks{\footnotesize $^{*}$Research supported in part by NSF grants DMS--0600814 and 
DMS--0901220}

\author[Torres--Ayala]{Francisco Torres--Ayala}

\begin{abstract}
We show that the polynomial $S_{m,k}(A,B)$, that is the sum of all words in noncommuting variables $A$ and $B$
having length $m$ and exactly $k$ letters equal to $B$, is
not equal to a sum of commutators and Hermitian squares in the algebra $\Reals\langle X,Y\rangle$,
where $X^2=A$ and $Y^2=B$,
for all even values of $m$ and $k$ with $6\le k\le m-10$, and also for $(m,k)=(12,6)$.
This leaves only the case $(m,k)=(16,8)$ open.
This topic is of interest in connection with the Lieb--Seiringer formulation of the Bessis--Moussa--Villani
conjecture, which asks whether $\Tr(S_{m,k}(A,B))\ge0$ holds for all positive semidefinite matrices $A$ and $B$.
These results eliminate the possibility of using ``descent + sum-of-squares''
to prove the BMV conjecture.

We also show that $S_{m,4}(A,B)$ is equal to a sum of commutators and Hermitian squares in $\Reals\langle A,B\rangle$
when $m$ is even and not a multiple of $4$,
which implies $\Tr(S_{m,4}(A,B))\ge0$ holds for all Hermitian matrices $A$ and $B$, for these values of $m$.
\end{abstract}

\subjclass[2000]{15A24 (82B10)}

\keywords{BMV conjecture, Hermitian squares}

\date{March 27, 2010}

\maketitle

\section{Introduction}
While working on quantum statistical mechanics,  Bessis, Moussa and Villani~\cite{BMV} conjectured in 1975
that for any positive semidefinite Hermitian
matrices $A$ and $B$, the function $t\mapsto\Tr(e^{A-tB})$ is the Laplace transform of a positive measure supported
in $\Reals_+$.
This is referred to as the Bessis--Moussa--Villani or BMV conjecture.
In 2004, Lieb and Seiringer~\cite{LS04} proved that the BMV conjecture is equivalent to the following
reformulation:
for every $A$ and $B$ as above, all of the coefficients of the polynomial
\begin{equation}\label{eq:pt}
p(t)=\Tr((A+tB)^m)\in\Reals[t]
\end{equation}
are nonnegative.
Recently, there has been much activity around this algebraic reformulation, (see~\cite{HJ05}, \cite{H07},
\cite{LaSp}, \cite{B}, \cite{KS08}).
The latest state of knowledge is summarized in~\cite{KS08},
and we'll review this here.

Let $S_{m,k}(A,B)$ denote
the sum of all words of length $m$ in $A$ and $B$
having $k$ letters equal to $B$ and $m-k$ equal to $A$.
Thus, the coefficient of $t^k$ in the polynomial $p(t)$ of~\eqref{eq:pt} is equal to the trace of $S_{m,k}(A,B)$,
and the Lieb--Seiringer reformulation of the BMV conjecture is that this trace is always nonnegative.
An important result, due to Hillar~\cite{H07}, is that if this conjecture fails for some $(m,k)$, then it fails for all
$(m',k')$ satisfying  $k'\ge k$ and $m'-k'\ge m-k$.
We'll refer to this as Hillar's descent theorem.

One strategy that has been used to show that the trace of $S_{m,k}(A,B)$ is nonnegative for certain values of $m$ and $k$
is to let $X$ and $Y$ be formal square roots of $A$ and $B$, respectively
and, working in the algebra $\Reals\langle X,Y\rangle$ of polynomials in noncommuting variables $X$ and $Y$,
to show that $S_{m,k}(A,B)$ is equal to a sum of commutators $[g,h]=gh-hg$ and Hermitian squares
$f^*f$.
Here, the algebra $\Reals\langle X,Y\rangle$ is endowed with the involutive $*$--operation that is anti--multiplicative
and so that $X=X^*$ and $Y=Y^*$ are Hermitian.
We adopt the notation of~\cite{KS08} and say that two elements $a,b\in\Reals\langle X,Y\rangle$ are
{\em cyclically equivalent}  (written $a\cycsim b$) if they differ by a sum of commutators.
We will use repeatedly Proposition~2.3 of~\cite{KS08}, which states that two words $v$ and $w$ in $X$ and $Y$ are cyclically equivalent
if and only if they can be written $v=u_1u_2$ and $w=u_2u_1$ for words $u_1$ and $u_2$ in $X$ and $Y$,
and that two polynomials $a,b\in\Reals\langle X,Y\rangle$ are cyclically equivalent
if and only if for each cyclic equivalence
class $[w]$ of words in $X$ and $Y$,
the sum over all $v$ in $[w]$ of the coefficients $a_v$ of $a$ 
agrees with the sum over all $v$ in $[w]$ of the coefficients $b_v$ of $b$.
It is clear that any element of $\Reals\langle A,B\rangle$ that is cyclically equivalent
in $\Reals\langle X,Y\rangle$ to a sum of Hermitian squares in $\Reals\langle X,Y\rangle$ must have nonnegative
trace whenever $A$ and $B$ are replaced by positive semidefinite matrices,
and this has been the strategy used to show that $S_{m,k}(A,B)$ has nonnegative trace, for certain values of $m$ and $k$.
We will adopt the terminology of~\cite{KS08} and write $\Theta^2$ to denote the
set of elements of $\Reals\langle X,Y\rangle$ that are cyclically equivalent to sums of Hermitian squares in
$\Reals\langle X,Y\rangle$.
(It is not difficult to see that $\Theta_\Cpx^2\cap\Reals\langle X,Y\rangle=\Theta^2$,
where $\Theta_\Cpx^2$ is the analogous quantity in $\Cpx\langle X,Y\rangle$.)

Clearly, $S_{m,k}(A,B)\in\Theta^2$ if and only if $S_{m,m-k}(A,B)\in\Theta^2$.
Due to work of H\"agele~\cite{Hae07},
Landweber and Speer~\cite{LaSp}, Burgdorf~\cite{B} and Klep and Schweighofer~\cite{KS08}, it is known that
$S_{m,k}(A,B)\in\Theta^2$ holds
\begin{enumerate}[$\bullet$]
\item whenever $k\in\{0,1,2,4\}$
\item for $m=14$ and $k=6$
\item for $m\in\{7,11\}$ and $k=3$
\end{enumerate}
These cases together with Hillar's descent theorem implied that the Lieb--Seiringer formulation of the BMV--conjecture
holds for $m\le13$ (see~\cite{KS08}). 
On the other hand, it is known that 
$S_{m,k}(A,B)\notin\Theta^2$ holds
\begin{enumerate}[$\bullet$]
\item whenever $m\ge12$ or $m\in\{6,8,9,10\}$ and $k=3$
\item whenever $m\ge10$ and $5\le k\le m-5$ and either $k$ or $m$ is odd.
\end{enumerate}
It was hoped that proofs of $S_{m,k}(A,B)\in\Theta^2$ for other values of $m$ and $k$
would be posible, so as to prove the conjecture for more values of $m$, and possibly even to prove
the BMV conjecture itself.

These results left open the cases $(m,k)=(12,6)$ and $m\ge16$, $6\le k\le m-6$ with both $m$ and $k$ even.
In this paper (see Section~\ref{sec:nonsos}),
we prove $S_{m,k}(A,B)\notin\Theta^2$ whenever $m$ and $k$ are even and $6\le k\le m-10$.
Using $S_{m,k}(A,B)=S_{m,m-k}(B,A)$, this leaves open only the cases $(m,k)=(12,6)$ and $(m,k)=(16,8)$.
We resolve the first of these cases by showing, via an easier argument, $S_{12,6}(A,B)\notin\Theta^2$.
The case of $(m,k)=(16,8)$ remains open,
though, as indicated in~\cite{KS08}, numerical evidence seems to suggest it does not lie in $\Theta^2$.

Our results, thus, show that it is impossible to prove the BMV conjecture by showing that $S_{m,k}(A,B)$ is cyclically
equivalent to a sum of Hermitian squares for sufficiently many values of $m$ and $k$.
However there are other plausible approaches to showing $\Tr(S_{m,k}(A,B))\ge0$
must always hold.

Though our proofs are straightforward and easy to check by hand, to find them we calculated with
Mathematica 7.0 \cite{M7.0}, on an Apple MacBook running OS X version 10.4.11.

\medskip

While exploring, we found
(see Proposition~\ref{prop:S4m+2,4}) that if $m$ is even and is not a multiple of $4$,
then $S_{m,4}(A,B)$ is equal to a sum of commutators and Hermitian squares in $\Reals\langle A,B\rangle$.
Thus, we do not need the square roots of $A$ and $B$:
for these values of $m$ we have $\Tr(S_{m,4}(A,B))\ge0$ whenever $A$ and $B$
are Hermitian matrices.
\begin{ques}\label{ques:H}
Do we have $\Tr(S_{m,k}(A,B))\ge0$ whenever $A$ and $B$
are Hermitian matrices and $m$ and $k$ are even integers, $m\ge k$?
\end{ques}
Using Hillar's descent theorem, a positive answer to Question~\ref{ques:H} would imply the Lieb--Seiringer
formulation of the BMV conjecture.

We will prove the following theorem in Section~\ref{sec:LagrangeMult}.
It shows that Question~\ref{ques:H} has an equivalent formulation that seems easier to satisfy,
and is analogous to Theorem~1.10 of~\cite{H07}.
Note that $S_{m,k}(A,B)$ is Hermitian whenever $A$ and $B$ are Hermitian.
\begin{thm}\label{thm:negdef}
Fix $n,m,k\in\Nats$ with $m$ and $k$ even and $m\ge k$.
Then the following are equivalent:
\begin{enumerate}[(i)]
\item
\label{it:negdef-i}
for all $n\times n$ Hermitian matrices $A$ and $B$, we have $\Tr(S_{m,k}(A,B))\ge0$,
\item
\label{it:negdef-ii}
for all $n\times n$ Hermitian matrices $A$ and $B$, either $S_{m,k}(A,B)=0$ or $S_{m,k}(A,B)$ has a strictly positive eigenvalue.
\end{enumerate}
\end{thm}

In Section~\ref{sec:Sm4} we also show (Proposition~\ref{prop:S84}) that
$S_{8,4}(A,B)$ is not cyclically equivalent to a sum of Hermitian squares in $\Reals\langle A,B\rangle$.
This makes the case $(m,k)=(8,4)$ of particular interest for Question~\ref{ques:H}.

Our interest in Question~\ref{ques:H} has two motivations.
One is its relation to the BMV conjecture.
Although the question is known to be stronger than the BMV conjecture and we have no particular reason to think it will be
easier to prove than the BMV conjecture itself, it is clearly related to the BMV conjecture and it
may be helpful to explore it.
A second motivation is the relation to Connes' embedding problem.
For positive semidefinite matrices
$A$ and $B$, the trace of $S_{6,3}(A,B)$ is always nonnegative, though it is not cyclically equivalent to a sum
of squares in $\Cpx\langle X,Y\rangle$;
as was pointed out in~\cite{KS08a}, this makes $S_{6,3}(A,B)$, with $A$ and $B$ positive operators in a II$_1$--factor,
an interesting test case for Connes' embedding problem.
In a similar way, if Question~\ref{ques:H} turns out to have a positive answer
for $S_{8,4}(A,B)$, then because of Proposition~\ref{prop:S84},
then it will provide another
interesting test case for Connes' embedding problem, involving self--adjoint operators.
At this point, it seems important to generate such test cases.

After a first version of this paper was circulated, we learned that S.\ Burgdorf (see Remarks~(b) and~(c)
of Section~4 of~\cite{B}) had, long previously to us, also found that if $m$ is not a multiple of $4$, then
$S_{m,4}(A,B)$ is cyclically equivalent to a sum of Hermitian squares
in $\Reals\langle A,B\rangle$;
no proof was given in~\cite{B}.

\medskip
\noindent {\em Acknowledgement}.
The authors thank an anonymous referee for suggestions that improved the exposition.

\section{Some non--sum--of--squares results}
\label{sec:nonsos}

In this section, we show that $S_{m,k}(A,B)$ is not cyclically equivalent to a sum of Hermitian
squares in $\Reals\langle X,Y\rangle$ for various values of $m$ and $k$, all of which are even.

Let $W_{q,p}(A,B)$ denote the set of all words in $A$ and $B$ containing $q$ $A$'s and $p$ $B$'s.
Let $Z$ denote the column vector whose entries are all words in $W_{\ell,k}(A,B)$ in some
fixed order, and similarly let $Z_X$ and, respectively, $Z_Y$ be column vectors containing
all elements of  $XW_{\ell-1,k}(A,B)X$, respectively, $YW_{\ell,k-1}(A,B)Y$.
Klep and Schweighofer have shown (Proposition 3.3 of~\cite{KS08}) that, for integers $k$ and $\ell$,
$S_{2(k+\ell),2k}(A,B)$ is cyclically equivalent
to a sum of Hermitian squares in $\Reals\langle X,Y\rangle$ if and only if there are real, positive
semidefinite matrices $H$, $H_X$ and $H_Y$ such that
\begin{equation}\label{eq:ZHZ}
Z^*HZ+Z_X^*H_XZ_X+Z_Y^*H_YZ_Y\cycsim S_{2(k+\ell),2k}(A,B),
\end{equation}
where $Z^*$ denotes the row vector whose entries are the adjoints of the entries of $Z$, {\em etc}.
Let us denote the matrix entry of $H$ corresonding to words $u,v\in W_{\ell,k}(A,B)$ by $H(u,v)$, and similarly
for $H_X$ and $H_Y$.
Thus, we have
\begin{equation}\label{eq:ZHZonly}
Z^*HZ=\sum_{u,v\in W_{\ell,k}(A,B)}H(u,v)u^*v,
\end{equation}
and similarly for the other two terms.

\begin{remark}\label{rem:Huvstar}
If $H$ is a matrix as appearing in~\eqref{eq:ZHZonly}, and if $\Hh$ is the matrix
defined by $\Hh(u,v)=H(u^*,v^*)$, then
\[
Z^*\Hh Z=\sum_{u,v}H(u^*,v^*)u^*v=\sum_{u,v}H(u,v)uv^*\cycsim\sum_{u,v}H(v,u)v^*u=Z^*HZ,
\]
where the last equality uses that $H$ is symmetric.
In a similar way, defining $\Hh_X(u,v)=\Hh_X(u^*,v^*)$ and
$\Hh_Y(u,v)=\Hh_Y(u^*,v^*)$, we have
\begin{align*}
Z_X^*\Hh_XZ_X&\cycsim Z_X^*H_XZ_X \\
Z_Y^*\Hh_YZ_Y&\cycsim Z_Y^*H_YZ_Y.
\end{align*}
Consequently, if $H$, $H_X$ and $H_Y$ are such that~\eqref{eq:ZHZ} holds, then by replacing $H$ with
$(H+\Hh)/2$, if necessary, and similarly for $H_X$ and $H_Y$, 
we may without loss of generality assume
\begin{alignat}{2}
H(u,v)&=H(u^*,v^*),\qquad&(u,v&\in W_{\ell,k}(A,B)), \label{eq:Huv} \\
H_X(u,v)&=H_X(u^*,v^*),\qquad&(u,v&\in XW_{\ell-1,k}(A,B)X), \label{eq:HXuv} \\
H_Y(u,v)&=H_Y(u^*,v^*),\qquad&(u,v&\in YW_{\ell,k-1}(A,B)Y). \label{eq:HYuv}
\end{alignat}

Suppose, furthermore, we have $k=\ell$.
Let $\sigma$ is the map on words that exchanges $A$ and $B$ and exchanges $X$ and $Y$, extended by linearity to
$\Reals\langle X,Y\rangle$.
Then $\sigma(Z^*HZ)=Z^*H^\sigma Z$, where $H^\sigma(u,v)=H(\sigma(u),\sigma(v))$, and, similarly,
$\sigma(Z_X^*H_XZ_X)=Z_Y^*H_X^\sigma Z_Y$ and $\sigma(Z_Y^*H_YZ_Y)=Z_X^* H_Y^\sigma Z_X$, where
$H^\sigma_X(u,v)=H_X(\sigma(u),\sigma(v))$ and $H^\sigma_Y(u,v)=H_Y(\sigma(u),\sigma(v))$.
Consequently, if $H$, $H_X$ and $H_Y$
are such that~\eqref{eq:ZHZ} holds, then since $S_{2(k+\ell),2k}(A,B)$ is $\sigma$--invariant and since $\sigma$
respects $\cycsim$, by replacing $H$ with
$(H+H^\sigma)/2$, $H_X$ with $(H_X+H_Y^\sigma)/2$ and $H_Y$ with $(H_Y+H_X^\sigma)/2$, if necessary,
we may without loss of generality assume
\begin{alignat}{2}
H(\sigma(u),\sigma(v))&=H(u,v),\qquad&(u,v&\in W_{\ell,k}(A,B)), \label{eq:Hsig} \\
H_Y(\sigma(u),\sigma(v))&=H_X(u,v),\qquad&(u,v&\in XW_{\ell-1,k}(A,B)X). \label{eq:HYsig} 
\end{alignat}
Since $\sigma(u^*)=\sigma(u)^*$, we can assume that~\eqref{eq:Huv}--\eqref{eq:HYuv}
and \eqref{eq:Hsig}--\eqref{eq:HYsig} hold simultaneously.

We note that the relation~\eqref{eq:Huv} will be used in this section, while~\eqref{eq:Hsig} will
be used only in the proof of Proposition~\ref{prop:S84}, and the conditions on $H_X$ and $H_Y$ won't be needed at all in this paper.
\end{remark}

\begin{remark}\label{rem:wuv}
For a given word $w\in W_{2\ell,2k}(A,B)$, we are interested in the different ways we can have
\begin{alignat}{2}
w&\cycsim u^*v,\qquad&(u,v&\in W_{\ell,k}(A,B)), \label{eq:wuv} \\
w&\cycsim u_X^*v_X,\qquad&(u_X,v_X&\in XW_{\ell-1,k}(A,B)X), \label{eq:wXuv} \\
w&\cycsim u_Y^*v_Y,\qquad&(u_Y,v_Y&\in YW_{\ell,k-1}(A,B)Y). \label{eq:wYuv}
\end{alignat}
Indeed, if $|[w]|$ denotes the number of different elements of $W_{2\ell,2k}(A,B)$ that are
cyclically equivalent to $w$, and assuming~\eqref{eq:ZHZ} holds, then we have
\begin{multline}\label{eq:Hsums}
|[w]|=\sum_{\{(u,v)\mid u^*v\cycsim w\}}H(u,v)
+\sum_{\{(u_X,v_X)\mid u_X^*v_X\cycsim w\}}H_X(u_X,v_X) \\
+\sum_{\{(u_Y,v_Y)\mid u_Y^*v_Y\cycsim w\}}H_Y(u_Y,v_Y)
\end{multline}
where the respective sums are over
all pairs $(u,v)$ such that~\eqref{eq:wuv} holds,
all pairs $(u_X,v_X)$ such that~\eqref{eq:wXuv} holds
and all pairs $(u_Y,v_Y)$ such that~\eqref{eq:wYuv} holds.
To find all the ways we have~\eqref{eq:wuv},
we can write down all the cyclic permutations of $w$ and record those
for which the first $k+\ell$ letters consists of $\ell$ $A$'s and $k$ $B$'s.
Furthermore, if we have an instance of~\eqref{eq:wXuv} with $u_X=Xu'X$ and $v_X=Xv'X$,
$u',v'\in W_{\ell-1,k}(A,B)$, then $w\cycsim X(u')^*Av'X\cycsim A(u')^*Av'$; this yields an instance
of~\eqref{eq:wuv}, where both $u^*$ and $v$ start with $A$, and clearly each such instance
corresponds in this manner to an instance of~\eqref{eq:wXuv}.
Similarly, the instances of~\eqref{eq:wYuv} are in one--to--one correspondence with those
instances of~\eqref{eq:wuv} where both $u^*$ and $v$ start with $B$.
\end{remark}

We will apply (in a finite dimensional setting) the following elementary lemma, whose proof we provide for completeness.
\begin{lemma}\label{lem:kerpos}
Let $\HEu=\HEu_1\oplus\HEu_2$ be an orthogonal direct sum decomposition of a Hilbert space
and let $T\in B(\HEu)$ be a positive operator: $T\ge0$.
With respect to the given decomposition of $\HEu$, write $T$ in block form
\[
T=\left(\begin{matrix}T_{11}&T_{12}\\T_{21}&T_{22}\end{matrix}\right),
\]
where $T_{ij}:\HEu_j\to\HEu_i$.
Suppose $v\in\ker T_{11}\subseteq\HEu_1$.
Then $v\in\ker T_{21}$.
\end{lemma}
\begin{proof}
If $T_{21}v\ne0$, then there is $w\in\HEu_2$ such that $\langle T_{21}v,w\rangle<0$.
Letting $t>0$ and using $T_{12}=T_{21}^*$, we have
\begin{equation}\label{eq:T21}
\langle T(v\oplus tw),v\oplus tw\rangle=2t\RealPart\langle T_{21}v,w\rangle+t^2\langle T_{22}w,w\rangle.
\end{equation}
But taking $t$ small enough forces the right--hand--side of~\eqref{eq:T21} to be negative,
which contradicts $T\ge0$.
\end{proof}

\begin{prop}\label{prop:Skl}
Let $k$ and $\ell$ be integers, $k\ge3$ and $\ell\ge5$.
Then $S_{2(\ell+k),2k}(A,B)$ is not cyclically equivalent to a sum of Hermitian squares
in $\Reals\langle X,Y\rangle$.
\end{prop}
\begin{proof}
Suppose the contrary, to obtain a contradiction.
Let $H$, $H_X$ and $H_Y$ be real, positive semidefinite matrices so that~\eqref{eq:ZHZ} holds,
and without loss of generality
assume also the property \eqref{eq:Huv} in Remark~\ref{rem:Huvstar} holds.
 
We consider five elements of $W_{2\ell,2k}(A,B)$ and the different ways of writing them as in~\eqref{eq:wuv}.
These elements are
\begin{alignat*}{2}
w_1&=A^{2\ell}B^{2k}\qquad&w_2&=A^{2\ell-2}B^{k-1}A^2B^{k+1} \\
w_3&=A^{\ell+1}B^2A^{\ell-1}B^{2k-2}\qquad&w_4&=A^{2\ell-4}B^{k-1}A^2B^2A^2B^{k-1} \\
w_5&=A^{\ell-1}B^2A^{\ell-1}B^{k-1}A^2B^{k-1}\qquad
\end{alignat*}
and their factorizations will be in terms of the elements
\begin{alignat*}{2}
u_1&=A^\ell B^k,\qquad&v_1=u_1^*&=B^kA^\ell \\
u_2&=A^{\ell-2}B^{k-1}A^2B,\qquad&v_2=u_2^*&=BA^2B^{k-1}A^{\ell-2} \\
u_3&=AB^2A^{\ell-1}B^{k-2},\qquad&v_3=u_3^*&=B^{k-2}A^{\ell-1}B^2A \\
u_4&=AB^{k-1}A^{\ell-1}B,\qquad&v_4=u_4^*&=BA^{\ell-1}B^{k-1}A
\end{alignat*}
of $W_{\ell,k}(A,B)$.
Note that these are all distinct if $k\ge4$;
in the case $k=3$, the six elements $u_1,u_2,u_3,v_1,v_2,v_3$ are distinct but we have $u_4=u_3$ and $v_4=v_3$.
This will not bother us.

We begin with the easiest of the $w_j$ to factorize, namely, $w_1$.
In the Table~\ref{tab:w1} are listed all the cyclically equivalent forms of $w_1$ and it is indicated which 
of these can be factored as in~\eqref{eq:wuv}.
\begin{table}[hb]
\caption{Forms of $w_1=A^{2\ell}B^{2k}$ and factorizations as in~\eqref{eq:wuv}.}
\label{tab:w1}
\begin{tabular}{l||l|l}
cyclically equivalent form & $j$ value & factorization \\\hline\hline
\rule{0ex}{2.5ex}$A^jB^{2k}A^{2\ell-j}\quad(1\le j\le 2\ell)$ & $j=\ell$ & $v_1^*v_1$ \\\hline
\rule{0ex}{2.5ex}$B^jA^{2\ell}B^{2k-j}\quad(1\le j\le 2k)$ & $j=k$ & $u_1^*u_1$ \\\hline
\end{tabular}
\end{table}

This also shows that there are no factorizations as in~\eqref{eq:wXuv} or~\eqref{eq:wYuv}
(see Remark~\ref{rem:wuv}).
Since $w_1$ has $2(k+\ell)$ cyclically equivalent forms,
by~\eqref{eq:Hsums}
we must have $H(u_1,u_1)+H(v_1,v_1)=2(k+\ell)$.
Since we have $H(v_1,v_1)=H(u_1,u_1)$, we get
\begin{equation}\label{eq:w1}
H(u_1,u_1)=k+\ell.
\end{equation}

The cyclically equivalent forms and all factorizations of $w_2$, $w_3$, $w_4$ and $w_5$ as in~\eqref{eq:wuv}
are given in Tables~\ref{tab:w2}--\ref{tab:w5}.
(Note that the assertions in rows 2, 3 and 6 of Table~\ref{tab:w4} do require $\ell\ge5$.)

\begin{table}[ht]
\caption{Forms of $w_2=A^{2\ell-2}B^{k-1}A^2B^{k+1}$ and factorizations as in~\eqref{eq:wuv}.}
\label{tab:w2}
\begin{tabular}{l||l|l}
\multicolumn{1}{c||}{cyclically equivalent form} & $j$ value & factorization \\
\hline\hline
\rule{0ex}{2.5ex}$A^jB^{k-1}A^2B^{k+1}A^{2\ell-2-j}\quad(1\le j\le 2\ell-2)$ & $j=\ell-2$ & $v_2^*v_1$ \\\hline
\rule{0ex}{2.5ex}$B^jA^2B^{k+1}A^{2\ell-2}B^{k-1-j}\quad(1\le j\le k-1)$ & \multicolumn{2}{c}{none} \\\hline
\rule{0ex}{2.5ex}$A^jB^{k+1}A^{2\ell-2}B^{k-1}A^{2-j}\quad(1\le j\le2)$ & \multicolumn{2}{c}{none} \\\hline
\rule{0ex}{2.5ex}$B^jA^{2\ell-2}B^{k-1}A^2B^{k+1-j}\quad(1\le j\le k+1)$ & $j=k$ & $u_1^*u_2$ \\\hline
\end{tabular}
\end{table}

\vskip2ex

\begin{table}[ht]
\caption{Forms of $w_3=A^{\ell+1}B^2A^{\ell-1}B^{2k-2}$ and factorizations as in~\eqref{eq:wuv}.}
\label{tab:w3}
\begin{tabular}{l||l|l}
\multicolumn{1}{c||}{cyclically equivalent form} & $j$ value & factorization \\
\hline\hline
\rule{0ex}{2.5ex}$A^jB^2A^{\ell-1}B^{2k-2}A^{\ell+1-j}\quad(1\le j\le\ell+1)$ & $j=1$ & $v_3^*v_1$ \\\hline
\rule{0ex}{2.5ex}$B^jA^{\ell-1}B^{2k-2}A^{\ell+1}B^{2-j}\quad(1\le j\le2)$ & \multicolumn{2}{c}{none} \\\hline
\rule{0ex}{2.5ex}$A^jB^{2k-2}A^{\ell+1}B^{2}A^{\ell-1-j}\quad(1\le j\le\ell-1)$ & \multicolumn{2}{c}{none} \\\hline
\rule{0ex}{2.5ex}$B^jA^{\ell+1}B^{2}A^{\ell-1}B^{2k-2-j}\quad(1\le j\le2k-2)$ & $j=k$ & $u_1^*u_3$ \\\hline
\end{tabular}
\end{table}

\vskip2ex

\begin{table}[ht]
\caption{Forms of $w_4=A^{2\ell-4}B^{k-1}A^2B^2A^2B^{k-1}$ and factorizations as in~\eqref{eq:wuv}.}
\label{tab:w4}
\begin{tabular}{l||l|l}
\multicolumn{1}{c||}{cyclically equivalent form} & $j$ value & factorization \\
\hline\hline
\rule{0ex}{2.5ex}$A^jB^{k-1}A^2B^2A^2B^{k-1}A^{2\ell-4-j}\quad(1\le j\le2\ell-4)$ & $j=\ell-2$ & $v_2^*v_2$ \\\hline
\rule{0ex}{2.5ex}$B^jA^2B^2A^2B^{k-1}A^{2\ell-4}B^{k-1-j}\quad(1\le j\le k-1)$ & \multicolumn{2}{c}{none} \\\hline
\rule{0ex}{2.5ex}$A^jB^2A^2B^{k-1}A^{2\ell-4}B^{k-1}A^{2-j}\quad(1\le j\le2)$ & \multicolumn{2}{c}{none} \\\hline
\rule{0ex}{2.5ex}$B^jA^2B^{k-1}A^{2\ell-4}B^{k-1}A^2B^{2-j}\quad(1\le j\le2)$ & $j=1$ & $u_2^*u_2$ \\\hline
\rule{0ex}{2.5ex}$A^jB^{k-1}A^{2\ell-4}B^{k-1}A^2B^2A^{2-j}\quad(1\le j\le2)$ & \multicolumn{2}{c}{none} \\\hline
\rule{0ex}{2.5ex}$B^jA^{2\ell-4}B^{k-1}A^2B^2A^2B^{k-1-j}\quad(1\le j\le k-1)$ & \multicolumn{2}{c}{none} \\\hline
\end{tabular}
\end{table}

\vskip2ex

\begin{table}[ht]
\caption{Forms of $w_5=A^{\ell-1}B^2A^{\ell-1}B^{k-1}A^2B^{k-1}$ and factorizations as in~\eqref{eq:wuv}.}
\label{tab:w5}
\begin{tabular}{l||l|l}
\multicolumn{1}{c||}{cyclically equivalent form} & $j$ value & factorization \\
\hline\hline
\rule{0ex}{2.5ex}$A^jB^2A^{\ell-1}B^{k-1}A^2B^{k-1}A^{\ell-1-j}\quad(1\le j\le\ell-1)$ & $j=1$ & $v_3^*v_2$ \\\hline
\rule{0ex}{2.5ex}$B^jA^{\ell-1}B^{k-1}A^2B^{k-1}A^{\ell-1}B^{2-j}\quad(1\le j\le2)$ & $j=1$ & $u_4^*u_4$ \\\hline
\rule{0ex}{2.5ex}$A^jB^{k-1}A^2B^{k-1}A^{\ell-1}B^2A^{\ell-1-j}\quad(1\le j\le\ell-1)$ & $j=\ell-2$ & $v_2^*v_3$ \\\hline
\rule{0ex}{2.5ex}$B^jA^2B^{k-1}A^{\ell-1}B^2A^{\ell-1}B^{k-1-j}\quad(1\le j\le k-1)$ & $j=1$ & $u_2^*u_3$ \\\hline
\rule{0ex}{2.5ex}$A^jB^{k-1}A^{\ell-1}B^2A^{\ell-1}B^{k-1}A^{2-j}\quad(1\le j\le2)$ & $j=1$ & $v_4^*v_4$ \\\hline
\rule{0ex}{2.5ex}$B^jA^{\ell-1}B^2A^{\ell-1}B^{k-1}A^2B^{k-1-j}\quad(1\le j\le k-1)$ & $j=k-2$ & $u_3^*u_2$ \\\hline
\end{tabular}
\end{table}

From these, we see that each of the words $w_j$, $2\le j\le 5$ has $2(k+\ell)$ different cyclically equivalent
forms, and none have factorizations
involving $X$ or $Y$, as in~\eqref{eq:wXuv} or~\eqref{eq:wYuv}.
Looking at the two factorizations of $w_2$, and using~\eqref{eq:Hsums} and $H(v_2,v_1)=H(v_1,v_2)=H(u_1,u_2)$, we conclude
\begin{equation}\label{eq:w2}
H(u_1,u_2)=k+\ell.
\end{equation}
Similarly, considering all the factorizations of $w_3$, $w_4$ and $w_5$ we get, respectively,
\begin{align}
H(u_1,u_3)&=k+\ell \label{eq:w3} \\
H(u_2,u_2)&=k+\ell \label{eq:w4} \\
2H(u_2,u_3)+H(u_4,u_4)&=k+\ell. \label{eq:w5}
\end{align}
Now from equations~\eqref{eq:w1}--\eqref{eq:w4}, for the
$3\times 3$ submatrix of $H$ corresponding to the entries $u_1,u_2,u_3$, we have
\begin{equation}\label{eq:Hu123}
\left(\begin{matrix}
 H(u_1,u_1)&H(u_1,u_2)&H(u_1,u_3) \\
 H(u_1,u_2)&H(u_2,u_2)&H(u_2,u_3) \\
 H(u_1,u_3)&H(u_2,u_3)&H(u_3,u_3)
 \end{matrix}\right)
=
\left(\begin{matrix}
 k+\ell&k+\ell&k+\ell \\
 k+\ell&k+\ell&H(u_2,u_3) \\
 k+\ell&H(u_2,u_3)&H(u_3,u_3)
\end{matrix}\right).
\end{equation}
From~\eqref{eq:Hu123}, the positivity of $H$ and Lemma~\ref{lem:kerpos}, we obtain also $H(u_2,u_3)=k+\ell$.
But then, from~\eqref{eq:w5}, we must have $H(u_4,u_4)=-(k+\ell)$, which contradicts the positive semidefiniteness of $H$.
\end{proof}

\begin{prop}\label{prop:S12,6}
$S_{12,6}(A,B)$ is not cyclically equivalent to a sum of squares in $\Reals\langle X,Y\rangle$.
\end{prop}
\begin{proof}
This is like the proof of Proposition~\ref{prop:Skl}, but easier.
Again we assume, to obtain a contradiction,
that $H$, $H_X$ and $H_Y$ are real, positive semidefinite matrices
such that~\eqref{eq:ZHZ} holds (with $k=\ell=3$) and that the properties~\eqref{eq:wuv}--\eqref{eq:wYuv} hold.
We need only consider the words
\[
w_6=A^6B^6,\qquad w_7=A^4B^2A^2B^4,\qquad w_8=A^2B^2A^2B^2A^2B^2
\]
in $W_{6,6}(A,B)$
and their factorizations, which will be in terms of the elements
\begin{alignat*}{2}
u_5&=A^3B^3,\qquad&v_5=u_5^*&=B^3A^3 \\
u_6&=AB^2A^2B,\qquad&v_6=u_6^*&=BA^2B^2A
\end{alignat*}
of $W_{3,3}(A,B)$.
These factorizations are given in Tables~\ref{tab:w6}--\ref{tab:w8}.
\begin{table}[ht]
\caption{Forms of $w_6=A^6B^6$ and factorizations as in~\eqref{eq:wuv}.}
\label{tab:w6}
\begin{tabular}{l||l|l}
cyclically equivalent form & $j$ value & factorization \\\hline\hline
\rule{0ex}{2.5ex}$A^jB^6A^{6-j}\quad(1\le j\le 6)$ & $j=3$ & $v_5^*v_5$ \\\hline
\rule{0ex}{2.5ex}$B^jA^6B^{6-j}\quad(1\le j\le 6)$ & $j=3$ & $u_5^*u_5$ \\\hline
\end{tabular}
\end{table}

\vskip2ex

\begin{table}[ht]
\caption{Forms of $w_7=A^4B^2A^2B^4$ and factorizations as in~\eqref{eq:wuv}.}
\label{tab:w7}
\begin{tabular}{l||l|l}
cyclically equivalent form & $j$ value & factorization \\\hline\hline
\rule{0ex}{2.5ex}$A^jB^2A^2B^4A^{4-j}\quad(1\le j\le 4)$ & $j=1$ & $v_6^*v_5$ \\\hline
\rule{0ex}{2.5ex}$B^jA^2B^4A^4B^{2-j}\quad(1\le j\le 2)$ & \multicolumn{2}{c}{none} \\\hline
\rule{0ex}{2.5ex}$A^jB^4A^4B^2A^{2-j}\quad(1\le j\le 2)$ & \multicolumn{2}{c}{none} \\\hline
\rule{0ex}{2.5ex}$B^jA^4B^2A^2B^{4-j}\quad(1\le j\le 4)$ & $j=3$ & $u_5^*u_6$ \\\hline
\end{tabular}
\end{table}

\vskip2ex

\begin{table}[ht]
\caption{Forms of $w_8=A^2B^2A^2B^2A^2B^2$ and factorizations as in~\eqref{eq:wuv}.}
\label{tab:w8}
\begin{tabular}{l||l|l}
cyclically equivalent form & $j$ value & factorization \\\hline\hline
\rule{0ex}{2.5ex}$A^jB^2A^2B^2A^2B^2A^{2-j}\quad(1\le j\le 2)$ & $j=1$ & $v_6^*v_6$ \\\hline
\rule{0ex}{2.5ex}$B^jA^2B^2A^2B^2A^2B^{2-j}\quad(1\le j\le 2)$ & $j=1$ & $u_6^*u_6$ \\\hline
\end{tabular}
\end{table}

Again, $w_6$, $w_7$ and $w_8$ have no factorizations as in~\eqref{eq:wXuv} or~\eqref{eq:wYuv}.
From Table~\ref{tab:w6}, we see that $w_6$ has $12$ distinct cyclically equivalent forms,
and since $H(u_5,u_5)=H(v_5,v_5)$, from~\eqref{eq:Hsums} we get
$H(u_5,u_5)=6$.
From Table~\ref{tab:w7} and $H(v_6,v_5)=H(u_6,u_5)=H(u_5,u_6)$, we get $H(u_5,u_6)=6$,
while from Table~\ref{tab:w8} we see that $w_8$ has only four distinct cyclically equivalent forms,
and we get $H(u_6,u_6)=2$.
The $2\times 2$ submatrix of $H$ corresponding to $\{u_5,u_6\}$ is, therefore,
\[
\begin{pmatrix}
H(u_5,u_5)&H(u_5,u_6) \\
H(u_6,u_5)&H(u_6,u_6)
\end{pmatrix}
=
\begin{pmatrix}
6&6\\
6&2
\end{pmatrix},
\]
which is not positive semidefinite.
This gives a contradiction.
\end{proof}

\section{Sums of squares in $\Reals\langle A,B\rangle$}
\label{sec:Sm4}

In this section, we prove some results related to Question~\ref{ques:H}.
As per the discussion in the introduction (see Proposition~2.3 of~\cite{KS08}),
we say $f,g\in\Reals\langle A,B\rangle$ are cyclically equivalent
if and only if $f-g$ is a sum of commutators of elements from $\Reals\langle A,B\rangle$.
This holds if and only if, for every word $w$ in $A$ and $B$, the sum over words $v$ that are cyclic permutations of $w$
of the coeefficients in $f$ of $v$ agrees with the same sum for $g$.

Clearly, if $S_{m,k}(A,B)$ is cyclically equivalent to a sum $\sum_i f_i^*f_i$
of Hermitian squares, for $f_i\in\Reals\langle A,B\rangle$, then Question~\ref{ques:H} has a positive answer for
this particular pair $(m,k)$.

Of course, $S_{2m,0}(A,B)=A^{2m}$ is a Hermitian square in $\Reals\langle A,B\rangle$, for every integer $m\ge0$.

Verification of the following two lemmas is straightforward.

\begin{lemma}\label{lem:S4m,2}
Let $m\in\Nats$.
Then
\[
S_{4m,2}(A,B)\cycsim mf_m^*f_m+ 2m\sum_{j=0}^{m-1}f_j^*f_j,
\]
where
\begin{align*}
f_0&=BA^{2m-1} \\
f_j&=A^{j-1}BA^{2m-j}+A^jBA^{2m-j-1},\qquad(1\le j\le m).
\end{align*}
\end{lemma}

\begin{lemma}\label{lem:S4m+2,2}
Let $m\in\Nats$.
Then
\[
S_{4m+2,2}(A,B)\cycsim(2m+1)\sum_{j=0}^mf_j^*f_j,
\]
where
\begin{align*}
f_0&=BA^{2m} \\
f_j&=A^{j-1}BA^{2m-j+1}+A^jBA^{2m-j},\qquad(1\le j\le m).
\end{align*}
\end{lemma}

The next proposition shows that $S_{2q,4}(A,B)$ is cyclically equivalent
to a sum of Hermitian squares in $\Reals\langle A,B\rangle$, when $q$ is odd.
Note that Klep and Schweighofer in Section~5 of~\cite{KS08} proved this in the case $q=7$.
In fact, we found the expression~\eqref{eq:S4m+2,4} below by exploration using Mathematica~\cite{M7.0}
and checked it by computation for all values of $m\le20$.
The best proof we could find, which is given below, turned out to be surprisingly intricate.

\begin{prop}\label{prop:S4m+2,4}
Let $m\in\Nats$.
Then
\begin{equation}\label{eq:S4m+2,4}
S_{4m+2,4}(A,B)\cycsim(2m+1)\sum_{p=0}^m f_p^*f_p,
\end{equation}
where
\begin{eqnarray*}
f_0&=&\sum_{s=0}^{2m-1}BA^{2m-s-1}BA^s, \\
f_p&=&\sum_{i=p-1}^p \sum_{s=p}^{2m-i-1} A^iBA^{2m-s-i-1}BA^s, \qquad (1\leq p \leq m-1) \\
f_m&=&A^{m-1}B^2A^m.
\end{eqnarray*}
\end{prop}
As before $W_{q,4}(A,B)$ denotes the set of all words in $A$ and $B$ with exactly $q$ $A$'s and four $B$'s.
Let $\Nats_0=\Nats\cup\{0\}$. For  $\iota=(\iota_1,\iota_2,\iota_3,\iota_4.\iota_5) \in \Nats_0^5$  let 
$$
E(\iota)=A^{\iota_1}BA^{\iota_2}BA^{\iota_3}BA^{\iota_4}BA^{\iota_5}
$$
and take
$$
I=\{\iota \in \Nats_0^5 \mid \iota_1+\iota_2+\iota_3+\iota_4+\iota_5=4m-2 \}.
$$
 
Note that the map $\iota \mapsto E(\iota)$ gives a bijection from I onto $W_{4m-2,4}(A,B)$.
With this notation we may write
$$
S_{4m+2,4}(A,B)=\sum_{\iota \in I} E(\iota).
$$

The proof of Proposition~\ref{prop:S4m+2,4} will use the following three lemmas.
The first of these is readily verified, and a proof will be omitted.

\begin{lemma}\label{canonicalorder}
Each word in $W_{4m-2,4}(A,B)$ is cyclically equivalent to a unique word of the form
$$
BA^{k_1}BA^{k_2}BA^{k_3}BA^{k_4}
$$
where $\kappa=(0,k_1,k_2,k_3,k_4)\in I$ satisfies  either

\begin{alignat}{3}
 k_1&\leq k_3&\quad&\textrm{and}\quad&k_2&<k_4 \label{typeI} \\
 &&&\textrm{or}\nonumber \\
 k_1&=k_3&&\leq&k_2&=k_4. \label{typeII}
\end{alignat}
\end{lemma}

We will call the words (or indices) described in (\ref{typeI}) and (\ref{typeII}) {\em canonically ordered} and  those of the form (\ref{typeI}) will be called {\em type I} while those given by  (\ref{typeII}) will be called {\em type II}. 
Since the first letter of any canonically ordered word is a $B$, canonically ordered words are parameterized by only four non-negative integers, and we'll frequently omit to write the first element of a canonically ordered index $\kappa$,
since it is always zero.

\begin{lemma}\label{cardoftypes}
\begin{eqnarray*}
\# \{ \kappa \in I \mid\kappa  \textrm{ is canonically ordered of type I}  \} &=& \frac{2m(2m-1)(2m+1)} {3} . \\
\# \{ \kappa \in I \mid \kappa \textrm{ is canonically ordered of type II}  \}&=& m.
\end{eqnarray*}
\end{lemma}

\begin{proof}
We recall that a partition of $n \in \Nats$ into $k$ parts is a $k$-tuple $(a_1,a_2,\cdots,a_k)$ such that 
$1\leq a_1 \leq a_2 \leq \cdots \leq a_k$ and $a_1+a_2+\cdots + a_k=n$. We denote it as $(a_1,a_2,\cdots,a_k) \vdash n$.

Consider the sets 
$$
B=\{ (a,b,a_1,a_2,b_1,b_2) \in \Nats^6 \mid a+b=4m+1, (a_1,a_2) \vdash a , (b_1,b_2) \vdash b \}
$$
and

$$
A=\{ \kappa \in I \mid\kappa  \textrm{ is canonically ordered of type I}  \}.
$$

Take the function from $A$ into $B$ given by
$$
(k_1,k_2,k_3,k_4) \mapsto (k_1+k_3+2,k_2+k_4+1,k_1+1,k_3+1,k_2+1,k_4).
$$

One can show this function is a bijection onto $B$. Thus,
\begin{eqnarray*}
\# A = \sum_{ \begin{subarray}{2} (a,b)\in \Nats^2 \\ a+b=4m+1 \end{subarray}}
       \bigg\lfloor \frac{a}{2}  \bigg\rfloor \bigg\lfloor \frac{b}{2}  \bigg\rfloor 
 =  \frac{2}{3}m(2m-1)(2m+1).
\end{eqnarray*}

Similarly, the function
$$
(k_1,k_2,k_3,k_4) \mapsto (k_1+1,k_2+1)
$$
is a bijection from $\{  \kappa \in I \mid \kappa \textrm{ is canonically ordered of type II}   \}$ onto the set
$\{ (a,b) \in \Nats^2\mid (a,b) \vdash 2m+1 \}$.
Hence
$$
\# \{ \kappa \in I \mid \kappa \textrm{ is canonically ordered of type II}  \}= \left\lfloor \frac{2m+1}{2} \right\rfloor= m.
$$
\end{proof}

The following lemma is easily verified by writing out the cyclically equivalent forms of words;
see Tables~\ref{tab:w1}--\ref{tab:w8} for other exercises of this sort.
\begin{lemma}\label{sumcoeffS}
Let  $w \in W_{4m-2,4}(A,B)$ be a canonically ordered word.
If $w$ is of type I,
then there are $4m+2$ words in $W_{4m-2,4}(A,B)$ that are cyclically equivalent to $w$,
while if $w$ is of type II,
then there are $2m+1$ words in $W_{4m-2,4}(A,B)$ that are cyclically equivalent to $w$.
\end{lemma}

\begin{proof}[Proof of Proposition~\ref{prop:S4m+2,4}]
Let $l=2m-1$.

For $g\in\Reals\langle A,B\rangle$ and $w$ a word in $A$ and $B$, we let $c_w(g)$ denote the coefficient of $w$ in $g$.
By Lemmas~\ref{canonicalorder} and~\ref{sumcoeffS} it will suffice to show, for every canonically ordered word
$w\in W_{4m-2,4}(A,B)$,
\begin{equation}\label{eq:cw}
\sum_{\{v\mid v \cycsim w\}} \sum_{p=0}^m  c_{v}(f_p^*f_p) =
\begin{cases}
2,&w\textrm{ of type I} \\
1,&w\textrm{ of type II}
\end{cases}
\end{equation}
{\em i.e.}, for each such $w$, there is only one representative in $\sum_{p=0}^m f_p^*f_p$ if $w$ is type II
and exactly two representatives if $w$ is type I.

We begin by taking a closer look at each $f_p^*f_p$.
We have 
\[
f_0^*f_0 = \sum_{0\leq s , t \leq l} A^sBA^{l-s}B^2A^{l-t}BA^t=\sum_{\iota \in I_0} E(\iota),
\]
where
\[
I_0=\{ \iota=(s,l-s,0,l-t,t)\mid 0\leq s , t \leq l \}
\]
and
for $1\leq p \leq m-1$,
\begin{eqnarray*}
f_p^*f_p &=&\sum_{p-1\leq i , j \leq p} \quad \sum_{ \begin{subarray}{l}  p\leq s \leq l-i \\  p\leq t \leq l-j \end{subarray}   } A^sBA^{l-i-s}BA^{i+j}BA^{l-j-t}BA^t \\ 
&=& \sum_{\iota \in I_p(p-1,p-1) } E(\iota ) + \sum_{\iota \in I_p(p-1,p) } E(\iota )+ \sum_{\iota \in I_p(p,p-1) } E(\iota ) + \sum_{\iota \in I_p(p,p) } E(\iota ),
\end{eqnarray*}
where
\[
I_p(i,j)=\{ \iota=( s,l-i-s,i+j,l-j-t,t )\mid p\leq s \leq l-i, p\leq t \leq l-j   \},
\]
while
\[
f_{m}^*f_m=\sum_{\iota \in I_m}E(\iota),
\]
where
\[
I_m=\{(m,0,2m-2,0,m)\}.
\]
We also write $I_0(0,0)=I_0$ and $I_m(m-1,m-1)=I_m$.

Let $J$ be the disjoint union 
$$
J_0 \sqcup \left( \bigsqcup_{p=1}^{m-1} \bigsqcup_{p-1\leq i,j \leq p} J_p(i,j)  \right) \sqcup J_m
$$ 
where each $J_p(i,j)$ is a copy of the corresponding $I_p(i,j)$ and similarly for $J_0=J_0(0,0)$ and $J_m=J_m(m-1,m-1)$.
Formally, given $0\leq p \leq m$ and $\max \{0,p-1 \}\leq i,j \leq \min \{ p, m-1 \}$, we set
\begin{eqnarray*}
J_p(i,j)=\{ (p,i,j, \iota   )\mid \iota \in I_p(i,j)  \}
\end{eqnarray*}
and we let $\alpha_p^{(i,j)}:I_p(i,j) \to J_p(i,j)$ be the bijection given by $\iota\mapsto(p,i,j,\iota)$.

Consider the function $O:I \to I$, where $O(\iota)$ is the index of the
canonically ordered word that is cyclically equivalent to $E(\iota)$. 
This function $O$ is explicitly given on $I_0$ and on each $I_p(i,j)$
($1\leq p \leq m-1$, $p-1 \leq i,j\leq p$) as follows.
For $\iota=(s,l-i-s,i+j,l-j-t,t)\in I_p(i,j)$ we have 
$$
O(\iota)=
\begin{cases}
U(i,j,s,t),&\begin{aligned}[t]\textrm{if }(i=j\textrm{ and }t>s)&\textrm{ or }(i>j\textrm{ and }t-1 > s) \\
                                                    &\textrm{ or }(j>i\textrm{ and }t>s-1)
             \end{aligned} \\ 
L(i,j,s,t),&\begin{aligned}[t]\textrm{if }(i=j\textrm{ and }t\leq s)&\textrm{ or }(i>j\textrm{ and }t \leq s-1) \\
                                                        &\textrm{ or }(j>i\textrm{ and }t \leq s-1),
             \end{aligned}
\end{cases}                                                
$$
where $U$ and $L$ are given by 
\begin{eqnarray*}
U(i,j,s,t)&=&(0,l,0,l) +\left(\begin{array}{cccc} 1 & 1  & 0& 0 \\ 0 & -1 &0 &-1  \\ 0 & 0 &1 &1 \\ -1 & 0 &-1 &0  \end{array} \right) \left( \begin{array}{c} i \\ j \\s \\t  \end{array}\right),  \\[2ex]
L(i,j,s,t)&=&(l,0,l,0)+\left(\begin{array}{cccc} -1 & 0 & -1 & 0 \\ 1 & 1 &0 &0 \\ 0 & -1 &0 &-1 \\ 0 & 0 & 1 &1  \end{array} \right) \left( \begin{array}{c} i \\ j \\s \\t \end{array}\right).
\end{eqnarray*}
The canonical form of an element of $J$ is naturally taken
to be the same as the canonical form of the element of I to which it corresponds and we denote the
``canonical form map'' also by $O:J \to I$.

We now work on proving~\eqref{eq:cw}. 
For $0\leq p \leq m-1$ define
$$
\iota_p=(p,l-2p,2p,l-2p,p)\in I_p(p,p).
$$
Then $O(\iota_p)=(l-2p,2p,l-2p,2p)$, which is of type II.
We will show that there are no other words of type II in $J$.
Since we have $m$ different values of $O$, Lemma \ref{cardoftypes} will imply~\eqref{eq:cw} in the case $w$ is of type~II.

Let $K=J \setminus \{ \alpha_p^{(p,p)}(\iota_p) \mid 0\leq p \leq m-1 \}$.
We will find a  partition of $K$ into two  sets,
$B$ and $C$, both with cardinality $2m(2m-1)(2m+1)/3$, and a bijection $\beta:B\to C$ such that
$O(\beta(\iota))=O(\iota)$ and check that $O$ restricted to $B$ is injective and its values are of type~I.
From this it will follow that~\eqref{eq:cw} holds in the case $w$ is of type~I, and this will complete the proof
of~\eqref{eq:cw} in the case $w$ is of type~II.

The partition and bijection are defined below in several parts.
In all cases, it is straightforward to check the identity $O(\beta(i))=O(i)$.
\begin{enumerate}[(i)]
\item  For  $0\leq p \leq m-1$ take
\begin{eqnarray*}
B_1(p)&=&I_{p+1}(p,p), \\
C_1(p)&=& \{ (s,l-p-s,2p,l-p-t,t ) \in I_p(p,p)\mid p+1\leq s,t\}.
\end{eqnarray*}
We notice $B_1(p)=C_1(p)$ for all $0\leq p \leq m-1$.
This identification is used to define the restriction of $\beta$ to $J_{p+1}(p,p)$
by $\beta \circ \alpha_{p+1}^{(p,p)}=\alpha_p^{(p,p)}$.
For $\iota=(s,l-p-s,2p,l-p-t,t )\in I_{p+1}(p,p)$ we have
$$
O(\iota)=\begin{cases}
(l-p-s,2p,l-p-t,s+t),&p+1\leq t \leq s \leq l-p \\
(2p,l-p-t,s+t,l-p-s),&p+1\leq s < t \leq l-p,
\end{cases}
$$
and this element is of type I.

Let $B_1=\bigcup_{p=0}^{m-1} \alpha_{p+1}^{(p,p)}(B_1(p))$ and $C_1=\bigcup_{p=0}^{m-1}\alpha_p^{(p,p)}(C_1(p))$.
We have
$$\# B_1=\sum_{p=0}^{m-1}(2(m-p)-1)^2 .$$
 
\item
For $1\leq p \leq m-1$, let 
\begin{eqnarray*}
B_2(p)&=&\{ (s,l-(p-1)-s,2p-1,l-p-t,t)\in I_p(p-1,p)\mid p+1\leq s  \}, \\
C_2(p)&=&\{ (\tilde{s},l-p-\tilde{s},2p-1,l-(p-1)-\tilde{t},\tilde{t})\in I_{p}(p,p-1)\mid p+1\leq \tilde{t} \}.
\end{eqnarray*}
For $\iota=(s,l-(p-1)-s,2p-1,l-p-t,t)\in B_2(p)$ let 
$$
 \beta(\alpha_p^{(p-1,p)}(\iota)) =\alpha_p^{(p,p-1)}(s-1,l-p-(s-1),2p-1,l-(p-1)-(t+1),t+1).
$$
Then $\beta:\alpha_p^{(p-1,p)}(B_2(p))\to \alpha_p^{(p,p-1)}(C_2(p))$ is a bijection and
a computation shows
\begin{multline*}
O(\beta(\alpha_p^{(p-1,p)}(\iota)))=O(\alpha_p^{(p-1,p)}(\iota)) \\
=\begin{cases}
 (l-p-s-1,2p-1,l-p-t,s+t),&p\leq t \leq s-1 \leq l-p, \\
 (2p-1,l-p-t,s+t,l-p-s-1),&p\leq s-1 < t \leq l-p
 \end{cases}
\end{multline*}
and this is a word of type~I.
Take
\[
B_2=\bigcup_{p=1}^{m-1}\alpha_p^{(p-1,p)}(B_2(p)),\qquad C_2=\bigcup_{p=1}^{m-1}\alpha_p^{(p,p-1)}(C_2(p)).
\]
By disjointness, we have
$$
\# B_2 =\sum_{p=1}^{m-1}(2(m-p))^2.
$$

\item
In $I_0(0,0)$, the cases $(s,t)=(0,l)$ and $(s,t)=(l,0)$ have  the same value under $O$, namely 
$(0,0,l,l)$,
which is type I.
Take 
\[
B_3=\{ \alpha_0^{(0,0)}(l,0,0,l,0)\},\qquad C_3=\{ \alpha_0^{(0,0)}(0,l,0,0,l)\}
\]
and let  $\beta( \alpha_0^{(0,0)}(l,0,0,l,0)) = \alpha_0^{(0,0)}(0,l,0,0,l)$.
 
\item
Consider the set
\[
B_4(0)=\{ (0,l,0,l-t,t) : 1 \leq t \leq l-1 \} \subset I_0(0,0).
\]
For $\iota=(0,l,0,l-t,t)\in B_4(0)$, take         
\[
\beta(\alpha_0^{(0,0)}(\iota))=   
\begin{cases}
\alpha_q^{(q,q)}(l-q,0,2q,l-2q,q),&l-t\textrm{ even, }q=\frac{l-t}{2}, \\
\alpha_q^{(q,q-1)}(l-q,0,2q-1,l-2q+1,q),&l-t\textrm{ odd, }q=\frac{l-t+1}{2}.
\end{cases}
\]
Let $B_4=\alpha_0^{(0,0)}(B_4(0))$ and let $C_4$ be the image of $B_4$ under $\beta$.
A direct computation shows 
\[
O(\beta(\alpha_0^{(0,0)}(\iota) ))=O(\alpha_0^{(0,0)}(\iota))=
(0,l-t,t,l),
\]
which is  type I.
We also have
$\# B_4=2(m-1)$.
 
\item
Consider the set
\[
B_5(0)=\{ (s,l-s,0,l,0)  :1 \leq s \leq l-1 \} \subset I_0(0,0).
\]
For $\iota=(s,l-s,0,l,0)\in B_5(0)$ define
\[
\beta(\alpha_0^{(0,0)}(\iota))=   
\begin{cases}
\alpha_{q}^{(q,q)}(q,l-q,2q,0,l-q),&l-s\textrm{ even, }q=\frac{l-s}{2}, \\
\alpha_q^{(q-1,q)}(q,l-2q+1,2q-1,0,l-q),&l-s\textrm{ odd, }q=\frac{l-s+1}{2}.
\end{cases}
\]
Let $C_5$ be the image of $B_5$ under $\beta$.
Then $\beta:B_5 \to C_5 $ is a bijection and
\[
O(\beta(\alpha_0^{(0,0)}(\iota)))=O(\alpha_0^{(0,0)}(\iota))
=(l-s,0,l,s)
\]
is of type I.
We also have $\# B_5=2(m-1)$.

\item
Let
\begin{align*}
B_6^1&=\bigcup_{p=1}^{m-1} \{ \alpha_p^{(p-1,p)}(p,l-2p+1,2p-1,l-p-t,t)  : p\leq t\leq l-p-1 \}, \\[2ex]
B_6^2&=\bigcup_{p=1}^{m-2} \{ \alpha_p^{(p,p)}(p,l-2p, 2p,l-p-t,t) : p+1\leq t\leq l-p-1 \}
\end{align*}
and let $B_6=B_6^1 \cup B_6^2$. 
For 
\begin{equation}\label{primera}
\eta= \alpha_p^{(p-1,p)}(p,l-2p+1,2p-1,l-p-t,t) \in B_6^1,
\end{equation}
let
\[
\beta(\eta) 
= 
\begin{cases}
\alpha_{q}^{(q,q)}(2m-2p-q,2p-1,2q ,l-2q,q),&p+t\textrm{ odd} \\
\alpha_q^{(q,q-1)}(2m-2p-q,2p-1,2q-1,l-2q+1,q ),&p+t\textrm{ even},
\end{cases}
\] 
where $q=m-\lfloor\frac{p+t+1}{2}\rfloor$.
For 
\begin{equation}\label{segunda}
\eta=\alpha_p^{(p,p)}(p,l-2p, 2p,l-p-t,t) \in B_6^2   
\end{equation}
let
\[
\beta(\eta)
=
\begin{cases}
\alpha_{q}^{(q,q)}(2m-2p-q-1,2p,2q,l-2q ,q),&p+t\textrm{ odd,} \\
\alpha_q^{(q,q-1)}(2m-2p-q-1,2p,2q-1,l-2q+1,q)&p+t\textrm{ even,}
\end{cases}
\]
where $q=m-\lfloor\frac{p+t+1}{2}\rfloor$.
Take $C_6$ to be the image of $B_6$ under $\beta$.
Then $\beta:B_6 \to C_6 $ is a bijection and
\[
O(\beta(\eta))=O(\eta)=
\begin{cases}
(2p-1,l-p-t+1,p+t,l-2p),&\eta\textrm{ as in \eqref{primera}} \\
(2p,lp-t,p+t,l-2p),&\eta\textrm{ as in \eqref{segunda}}
\end{cases}
\]
is of type I.
We also have
\[
\# B_6=\sum_{p=1}^{m-1}(2(m-p)-1)+ \sum_{p=1}^{m-2}(2(m-p)-2)=(2m-3)(m-1).
\]
\end{enumerate}

Lastly, we take
\[
B=\bigsqcup_{k=1}^6 B_k,\qquad C=\bigsqcup_{k=1}^6 C_k.
\]
A computation shows  
\begin{align*}
\# B &=\sum_{p=1}^{m-1}(2(m-p)-1)^2 +  \sum_{p=1}^{m-1}(2(m-p))^2+1 
+ 4(m-1)+(2m-3)(m-1)\\
&=1+(2m-1)^2+4(m-1)+(2m-3)(m-1) + \sum_{j=1}^{2(m-1)}j^2 \\
&=\frac{2m(2m-1)(2m+1)}{3}.
\end{align*}

We have, thus, constructed a  bijection $\beta:B \to C$ that satisfies $O(\beta(\eta))=O(\eta)$ and,
as can be checked,
the restriction of $O$ to $B$ is injective and takes values that are all of type I.
Lastly the sets $B$ and $C$ form a partition of $K$.
This completes the proof of Proposition~\ref{prop:S4m+2,4}.

The bijection we have defined may be better understood using some pictures, which are contained in Figures~\ref{fig:first6}
and~\ref{fig:last4}.
We parameterize $I_0$ by the  square  $ \{ (s,t) \in \mathbb{Z}^2: 0\leq s,t\leq l \}$ and $I_m$ by the single point $(m,m)$.
Likewise for fixed $1\leq p \leq m-1$ and $i,j\in\{p-1,p\}$, the set $I_p(i,j)$ is parameterized by
$\{(s,t)\in \mathbb{Z}^2 : p\leq s \leq l-i, p\leq t \leq l-i \}$.
We show the case $m=3$.
\begin{figure}[hb]
\caption{Some sets in $K$ with $m=3$}
\label{fig:first6}
 \includegraphics{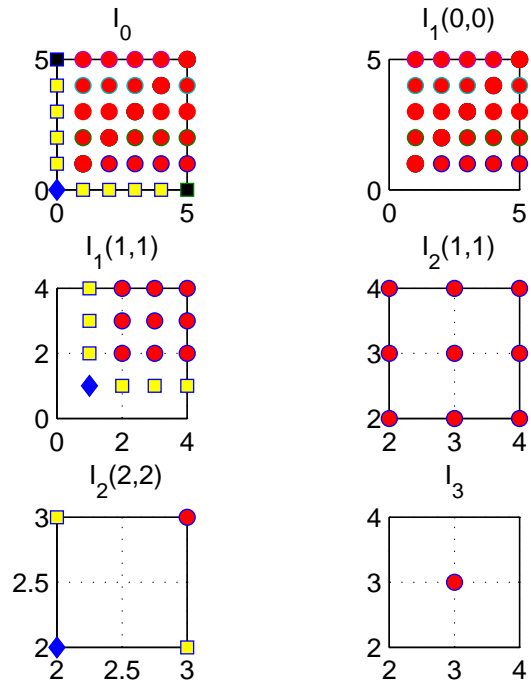}
\end{figure}

\begin{figure}
\caption{More sets in $K$ with $m=3$}
\label{fig:last4}
 \includegraphics{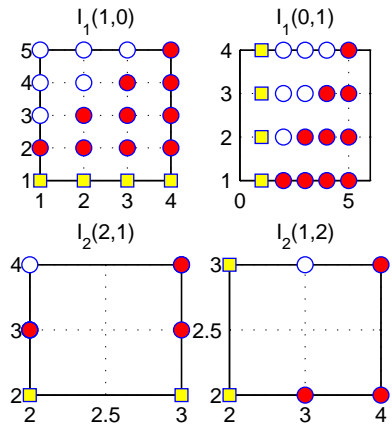}
\end{figure}

In these figures,
\begin{enumerate}[$\bullet$]
\item The points that give words of type II are marked with diamonds.
\item The light circles in the right column are matched with the circles in the left.
Likewise the solid circles.
These correspond to cases 1 and 2. 

In the case 2 the bijection is implemented by $(s,t)\mapsto (s-1,t+1) $, form the rightmost sub-square of side $l-2p+1$ in $I_p(p-1,p)$ to the uppermost sub-square of side $l-2p+1$ in $I_P(p,p-1)$, for $1\leq p\leq m-1$.

\item Case 3 is marked with a solid square.
\item  The remaining  points (which correspond to the most complicated part of the bijection), plotted in  light  squares, correspond the the cases 4,5 and 6.
\end{enumerate}


\end{proof}




The following theorem summarizes the results obtained so far in this section.
\begin{thm}
If $k=2$ and $m\ge2$ is even, or if $k=4$ and $m\ge6$ is even but not a multiple of $4$, then
$S_{m,k}(A,B)$ is cyclically equivalent to a sum of Hermitian squares in $\Reals\langle A,B\rangle$.
Therefore, for these values of $m$ and $k$, $\Tr(S_{m,k}(A,B))\ge0$ whenever $A$ and $B$ are Hermitian matrices.
\end{thm}

Below is a non-sum-of-squares result for $S_{8,4}(A,B)$.
However, Question~\ref{ques:H} for $m=8$ and $k=4$ is still open.

\begin{prop}\label{prop:S84}
The polynomial $S_{8,4}(A,B)$ is not cyclically equivalent to a sum of Hermitian squares in $\Reals\langle A,B\rangle$.  
\end{prop}
\begin{proof}
We order the elements of $W_{2,2}(A,B)$ in the column vector
\[
Z=(A^2B^2,ABAB,AB^2A,BA^2B,BABA,B^2A^2)^t.
\]
If $S_{8,4}(A,B)$ were equivalent to a sums of squares in $\Reals\langle A,B\rangle$,
then by Proposition~3.3 of~\cite{KS08}, we would have $S_{8,4}(A,B)\cycsim Z^*HZ$
for $H$ a $6\times6$ real, positive semidefinite matrix.
So suppose, to obtain a contradiction, that such exists.
There are ten cyclic equivalence class of words in $W_{4,4}(A,B)$.
We've chosen one representative for each and we have listed them in Table~\ref{tab:ws} with their orders,
where we say the {\em order} of a word is the number of cyclically equivalent forms that it has.
\begin{table}[hb]
\caption{Representatives of cyclic equivalence classes in $W_{4,4}(A,B)$.}
\label{tab:ws}
\begin{tabular}{c|l|c}
name & word & order \\
\hline\hline
$w_1$ & $A^4B^4$  & 8 \\ \hline
$w_2$ & $A^3BAB^3$ & 8 \\ \hline
$w_3$ & $A^3B^2AB^2$ & 8 \\ \hline
$w_4$ & $A^3B^3AB$ & 8 \\ \hline
$w_5$ & $A^2BA^2B^3$ & 8 \\ \hline
$w_6$ & $A^2BABAB^2$ & 8 \\ \hline
$w_7$ & $A^2BAB^2AB$ & 8 \\ \hline
$w_8$ & $A^2B^2A^2B^2$ & 4 \\\hline
$w_9$ & $A^2B^2ABAB$ & 8 \\ \hline
$w_{10}$ & $ABABABAB$ & 2
\end{tabular}
\end{table}
If we denote the $i$th element of the vector $Z$ by $z_i$, then the matrix whose $(i,j)$th entry
is the symbol $k\in\{1,\ldots,10\}$ such that $w_k$ is cyclically equivalent to $z_i^*z_j$
is the matrix found below.
\[
\begin{pmatrix}
 1 & 2 & 3 & 5 & 6 & 8 \\
 4 & 7 & 9 & 9 & 10 & 6 \\
 3 & 6 & 8 & 7 & 9 & 3 \\
 5 & 6 & 7 & 8 & 9 & 5 \\
 9 & 10 & 6 & 6 & 7 & 2 \\
 8 & 9 & 3 & 5 & 4 & 1
\end{pmatrix}.
\]
The hypothesis $Z^*HZ\sim S_{8,4}(A,B)$ is, therefore, equivalent to the ten linear equations
\begin{align}
8&=H_{11}+H_{66} \label{eq:Hij1} \\
8&=H_{12}+H_{56} \displaybreak[1] \\
8&=H_{13}+H_{31}+H_{36}+H_{63} \displaybreak[1] \\
8&=H_{21}+H_{65} \displaybreak[2] \\
8&=H_{14}+H_{41}+H_{46}+H_{64} \displaybreak[2] \\
8&=H_{15}+H_{26}+H_{32}+H_{42}+H_{53}+H_{54} \displaybreak[2] \\
8&=H_{22}+H_{34}+H_{43}+H_{55} \displaybreak[1] \\
4&=H_{16}+H_{33}+H_{44}+H_{61} \displaybreak[1] \\
8&=H_{23}+H_{24}+H_{35}+H_{45}+H_{51}+H_{62} \\
2&=H_{25}+H_{52} \label{eq:Hij10}
\end{align}
in the entries of the  matrix $H$.
However, $H$ is real symmetric.
Moreover, we may assume without loss of generality 
that the relations~\eqref{eq:Huv}
and~\eqref{eq:Hsig}
from Remark~\ref{rem:Huvstar} hold, and we find, therefore, that $H$ commutes with
the permutation matrices corresponding to the order--two permutations
\begin{align*}
\tau&\;:\;1\leftrightarrow6,\;2\leftrightarrow5. \\
\sigma&\;:\;1\leftrightarrow6,\;2\leftrightarrow5,\;3\leftrightarrow4.
\end{align*}
Thus, we have
\[
H=
\begin{pmatrix}
H_{11} & H_{12} & H_{13} & H_{13} & H_{15} & H_{16} \\
H_{12} & H_{22} & H_{23} & H_{23} & H_{25} & H_{15} \\
H_{13} & H_{23} & H_{33} & H_{34} & H_{23} & H_{13} \\
H_{13} & H_{23} & H_{34} & H_{33} & H_{23} & H_{13} \\
H_{15} & H_{25} & H_{23} & H_{23} & H_{22} & H_{12} \\
H_{16} & H_{15} & H_{13} & H_{13} & H_{12} & H_{11}
\end{pmatrix}.
\]
The equations~\eqref{eq:Hij1}--\eqref{eq:Hij10} now yield several relations, for example, from~\eqref{eq:Hij1} we get
$H_{11}=4$.
Using these relations to eliminate some variables, we have that $H$ equals the matrix
\[
\begin{pmatrix}
4 & 4 & 2 & 2 & 4-2H_{23} & 2-H_{33} \\
4 & H_{22} & H_{23} & H_{23} & 1 & 4-2H_{23} \\
2 & H_{23} & H_{33} & 4-H_{22} & H_{23} & 2 \\
2 & H_{23} & 4-H_{22} & H_{33} & H_{23} & 2 \\
4-2H_{23} & 1 & H_{23} & H_{23} & H_{22} & 4 \\
2-H_{33} & 4-2H_{23} & 2 & 2 & 4 & 4
\end{pmatrix}.
\]
We will show that there is no positive semidefinite real matrix of this form.
To make the formulas slightly more readable, we will use the symbols
$x_2=H_{22}$ and $x_3=H_{33}$.
Of course, we must have $x_2\ge0$ and $x_3\ge0$.
We will consider compressions of $H$ obtained by restricting to rows and columns in subsets of $\{1,\ldots,6\}$.
The compression to $\{1,2\}$ is $\left(\begin{smallmatrix}4&4\\4&x_2\end{smallmatrix}\right)$,
and from positivity we obtain $x_2\ge4$.
Compression to $\{1,6\}$ yields $|2-x_3|\le4$, so $x_3\le6$.
Compression to $\{1,3\}$ yields $x_3\ge1$.
The determinant of the compression of the matrix $H$ to $\{1,3,4,6\}$ is the polynomial with factorization
\[
(2+x_3)(x_2+x_3-4)(8-6x_2+2x_3+x_2x_3-x_3^2).
\]
Since $x_3\ge1$ and $x_2\ge4$, the first two factors are strictly positive.
So the third factor must be nonnegative, and we conclude
\[
x_2(x_3-6)\ge(x_3-4)(x_3+2).
\]
Since $x_3\le6$ we must have $x_3\le4$ and
\[
x_2\le\frac{(4-x_3)(x_3+2)}{6-x_3}\,.
\]
But combining this with $x_2\ge4$, we get
$24-4x_3\le8+2x_3-x_3^2$, so $x_3^2-6x_3+16\le0$, which is impossible.
This is the desired contradiction.
\end{proof}

\section{Proof of Theorem~\ref{thm:negdef}}
\label{sec:LagrangeMult}

In this section, we prove Theorem~\ref{thm:negdef} using a straightforward application of the method
of Lagrange multipliers.

\begin{lemma}\label{lem:LagrangeMult}
Let $n,m,k\in\Nats$ and fix an $n\times n$ Hermitian matrix $B$.
Consider the function $A\mapsto\Tr(S_{m,k}(A,B))$
with domain consisting of the $n\times n$ Hermitian matrices $A$ such that $\Tr(A^2)=1$.
Suppose $A_0$ is a point where this function has a relative extrumum.
Then
\begin{equation}\label{eq:SmA}
S_{m-1,k}(A_0,B)=\frac{m-k}m\Tr(S_{m,k}(A_0,B))A_0.
\end{equation}
\end{lemma}
\begin{proof}
This is an application of the method of Lagrange multipliers to the problem 
of optimizing $\Tr(S_{m,k}(A,B))$ subject to the constraint $\Tr(A^2)=1$.
(Compare to Appendix~A of~\cite{KS08}.)
The space of Hermitian $n\times n$ matrices is a real vector space of dimension $n^2$.
If $H$ and $A$ are Hermitian matrices, then
\begin{equation}\label{eq:Hconstr}
\frac d{d\lambda}\bigg|_{\lambda=0}\Tr((A+\lambda H)^2)=2\Tr(HA).
\end{equation}
Letting $H$ run through a fixed basis for the space of $n\times n$ Hermitian matrices,
the list of values~\eqref{eq:Hconstr} forms the gradient of the constraint function with respect to the
$n^2$ variables.

Letting $W_{m-k,k}(A,B)$ be the set of all words in noncommuting variables $A$ and $B$ with $m-k$ $A$'s and $k$ $B$'s,
we have $|W_{m-k,k}(A,B)|=\binom{m}{k}$.
If $w=w(A,B)\in W_{m-k,k}(A,B)$, then
$\frac d{d\lambda}\big|_{\lambda=0}w(A+\lambda H,B)$ equals the sum of the $m-k$ words obtained by replacing in turn and individually
the letters of $w$ that are equal to $A$ by $H$.
Thus, $\frac d{d\lambda}\big|_{\lambda=0}S_{m,k}(A+\lambda H,B)$ is the sum of all $(m-k)\binom{m}{k}$ words in $A$, $B$ and $H$, 
where $A$ appears $m-k-1$ times, $B$ appears $k$ times and $H$ appears once.
Taking the trace,
we get
\begin{equation}\label{eq:Hobj}
\frac d{d\lambda}\bigg|_{\lambda=0}\Tr(S_{m,k}(A+\lambda H,B))=m\Tr(HS_{m-1,k}(A,B)).
\end{equation}
Letting $H$ run through the same basis as taken above, the list of values~\eqref{eq:Hobj}
forms the gradient of the objective function with respect to the $n^2$ variables.

By the method of Lagrange multipliers, we conclude that at a relative extremum $A_0$, these two gradients
must be parallel.
In other words, we must have
\[
2\mu\Tr(HA_0)=m\Tr(HS_{m-1,k}(A_0,B))
\]
for some $\mu\in\Reals$ and all $H$, and this implies
\[
2\mu A_0=mS_{m-1,k}(A_0,B)
\]
Multiplying both sides by $A_0$, taking the trace and using Lemma~2.1 of~\cite{H07}, we get
\[
2\mu=2\mu\Tr(A_0^2)=m\Tr(A_0S_{m-1,k}(A_0,B))=(m-k)\Tr(S_{m,k}(A_0,B)),
\]
and~\eqref{eq:SmA} follows.
\end{proof}

\begin{proof}[Proof of Theorem~\ref{thm:negdef}.]
The implication~(\ref{it:negdef-i})$\implies$(\ref{it:negdef-ii}) is clear.

Suppose~(\ref{it:negdef-i}) does not hold.
Let $A_0$ and $B_0$ be a Hermitian $n\times n$ matrices
where $\Tr(S_{m,k}(A,B))$ takes its absolute minimum subject to $\Tr(A^2)=\Tr(B^2)=1$.
By assumption, we have $\Tr(S_{m,k}(A_0,B_0))<0$.
By Lemma~\ref{lem:LagrangeMult} and the analogue obtained by switching $A$ and $B$, we have
\begin{align*}
S_{m-1,k}(A_0,B_0)&=\frac{m-k}m\Tr(S_{m,k}(A_0,B_0))A_0 \\
S_{m,k-1}(A_0,B_0)&=\frac km\Tr(S_{m,k}(A_0,B_0))B_0.
\end{align*}
Therefore, the Hermitian matrix
\begin{multline*}
S_{m,k}(A_0,B_0)=A_0S_{m-1,k}(A_0,B_0)+B_0S_{m-1,k-1}(A_0,B_0) \\
=\Tr(S_{m,k}(A_0,B_0))\bigg(\frac{m-k}mA_0^2+\frac kmB_0^2\bigg)
\end{multline*}
has only nonpositive eigenvalues.
Thus, (\ref{it:negdef-ii}) does not hold.
\end{proof}

\bibliographystyle{plain}

\end{document}